\documentclass[10pt,a4paper]{article}

\usepackage[T1]{fontenc}
\usepackage[bitstream-charter]{mathdesign}

\usepackage[latin1]{inputenc}							
\usepackage{amsmath}									

\usepackage{xcolor, setspace}
\definecolor{bl}{rgb}{0.0,0.2,0.6}

\usepackage{sectsty}
\usepackage[compact]{titlesec}
\usepackage{tabularx}

\usepackage{amsthm, amsmath}
\usepackage{enumerate}

\usepackage[all]{xy}

\usepackage{graphics}
\usepackage{graphicx}
\usepackage{enumerate}

\newtheorem*{thma}{Theorem A}
\newtheorem*{thmb}{Theorem B}

\newtheorem{theorem}{Theorem}[section]
\newtheorem{definition}[theorem]{Definition}
\newtheorem{lemma}[theorem]{Lemma}

\newtheorem{proposition}[theorem]{Proposition}
\newtheorem{corollary}[theorem]{Corollary}
\newtheorem*{corollary*}{Corollary}

\title{Classifying spaces with virtually cyclic stabilizers for linear groups }

\author{Dieter Degrijse\footnote{supported by the Danish National Research Foundation through the Centre for Symmetry and Deformation (DNRF92)}, Ralf K\"{o}hl and Nansen Petrosyan}

\newcommand{\Z}{\mathbb Z}
\newcommand{\G}{\Gamma}

\medskip

\begin{document}

\maketitle
\begin{abstract} We show that every discrete subgroup of $\mathrm{GL}(n,\mathbb{R})$ admits a finite dimensional classifying space with virtually cyclic stabilizers. Applying our methods to $\mathrm{SL}(3,\mathbb{Z})$, we obtain a four dimensional classifying space with virtually cyclic stabilizers and a decomposition of the algebraic $K$-theory of its group ring.

\end{abstract}
\section{Introduction}
A classifying space of a discrete group $\Gamma$ for a family of subgroups $\mathcal{F}$ is a $\Gamma$-CW complex $X$ with stabilizers in $\mathcal{F}$ such that $X^H$  is contractible for every $H \in \mathcal{F}$. Such a space is also called a model for $E_{\mathcal{F}}\Gamma$.  A model for $E_{\mathcal{F}}\Gamma$  always exists for any given discrete group $\Gamma$ and a family of subgroups $\mathcal{F}$, but it need not be of finite type or finite dimensional (see \cite{Luck2}). The smallest possible dimension of a model for $E_{\mathcal{F}}\Gamma$ is the geometric dimension of $\Gamma$ for the family $\mathcal{F}$, denoted by $\mathrm{gd}_{\mathcal{F}}(\Gamma)$. When $\mathcal{F}$ is the family of finite, respectively, virtually cyclic subgroups of $\Gamma$, $E_{\mathcal{F}}\Gamma$ ($\mathrm{gd}_{\mathcal{F}}(\Gamma)$) is denoted by $\underline{E}\Gamma$ ($\underline{\mathrm{gd}}(\Gamma)$), respectively, $\underline{\underline{E}}\Gamma$ ($\underline{\underline{\mathrm{gd}}}(\Gamma)$).  

For any group $\G$, one always has ${\underline{\mathrm{gd}}}(\Gamma)\leq \underline{\underline{\mathrm{gd}}}(\Gamma)+1$ (see \cite{LuckWeiermann}).
In all examples known so far, a group $\G$ admits a finite dimensional model for $\underline{\underline{E}}\G$ if it admits a finite dimensional model for $\underline{E}\G$. However, it is still an open problem whether this is always the case. It is known that the invariant $\underline{\underline{\mathrm{gd}}}(\Gamma)$ can be arbitrarily larger than $\underline{\mathrm{gd}}(\Gamma)$ (see \cite{DP}). 

Questions concerning finiteness properties of ${\underline{E}}\Gamma$  and $\underline{\underline{E}}\Gamma$ have been especially motivated by the Farrell--Jones isomorphism conjecture in K- and L-theory (see below and \cite{DL, Lafont, LuckReich}).
Finding models for $\underline{\underline{E}}\Gamma$ with good finiteness properties has been proven to be much more difficult than for ${\underline{E}}\Gamma$. So far, such models have
 been found for polycyclic-by-finite groups~\cite{LuckWeiermann}, word-hyperbolic groups \cite{LearyPineda}, relatively hyperbolic groups \cite{LafontOrtiz},  countable elementary amenable group of finite Hirsch length \cite{DP1, DP} and groups acting isometrically with discrete orbits on separable complete  CAT(0)-spaces, such as mapping class groups and finitely generated linear groups over fields of positive characteristic \cite{DP2, Luck3}.

Here, we will show that certain subgroups of $\mathrm{GL}(n,\mathbb{C})$ admit a finite dimensional classifying space with virtually cyclic stabilizers.

\begin{thma}\label{th: main} Let $\Gamma$ be a discrete subgroup of $\mathrm{GL}(n,\mathbb{R})$ such that the Zariski closure of $\G$ in $\mathrm{GL}(n,\mathbb{R})$ has dimension $m$. Then $\Gamma$ admits a model for $\underline{\underline{E}}\Gamma$ of dimension $m+1$.
\end{thma}

Recall that a subgroup $\G$ of $\mathrm{GL}(n,\mathbb C)$ is said to be of {\it integral characteristic} if the coefficients of the characteristic polynomial of  every element of $\G$ are algebraic integers. It follows that $\G$ has integral characteristic if and only if the characteristic roots of every element of $\G$ are algebraic integers (see \cite[\S 2]{AS}). The standard embedding $\G \hookrightarrow \mathrm{GL}(n,\mathbb{C}) \hookrightarrow \mathrm{SL}(n+1,\mathbb{C})$ allows one to consider $\G$ as a subgroup of integral characteristic of $\mathrm{SL}(n+1,\mathbb{C})$.

\begin{thmb}\label{th: follows} Let $\Gamma$ be a finitely generated subgroup of $\mathrm{GL}(n,\mathbb{C})$ of integral characteristic such that there is an upper bound on the Hirsch lengths of its finitely generated unipotent subgroups.  Then $\Gamma$ admits a finite dimensional  model for $\underline{\underline{E}}\Gamma$.
\end{thmb}

\begin{corollary*}
Let $\mathbb F$ be an algebraic number field and suppose $\Gamma$ is a subgroup of  $\mathrm{GL}(n,\mathbb{F})$ of integral characteristic. 
Then $\Gamma$ admits a finite dimensional  model for $\underline{\underline{E}}\Gamma$.
\end{corollary*}

\noindent Theorems A and B and the Corollary will be proven in Section 4.\\

The $K$-theoretical Farrell--Jones conjecture (e.g. see \cite{DL, LuckReich}) predicts that for a group $\Gamma$ and a ring $R$, the assembly map
\[     \mathcal{H}^{\G}_n(\underline{\underline{E}}\Gamma;\mathbf{K}_R) \rightarrow   \mathcal{H}^\G_n(\{\ast\};\mathbf{K})=\mathrm{K}_n(R[\G])  \]
is an isomorphism for every $n \in \mathbb{Z}$. Here $\mathrm{K}_{\ast}(R[\Gamma])$ is the algebraic $K$-theory of the group ring $R[\Gamma]$ and $ \mathcal{H}^{\G}_*(-;\mathbf{K}_R) $ is a generalized equivariant homology theory defined using the $K$-theory spectrum $\mathbf{K}_R$. This conjecture has been proven for many important classes of groups (and rings), including $\mathrm{SL}(n,\mathbb{Z})$ when $R$ is finitely generated as an abelian group (see \cite{BartelsLuckReichRueping}). 

Using the universal property of classifying spaces for families, one can construct a $\Gamma$-equivariant inclusion of $\underline{E}\Gamma$ into $\underline{\underline{E}}\Gamma$. By a result of Bartels (see \cite[Th. 1.3.]{Bartels}) this inclusion induces a split injection $\mathcal{H}^{\G}_n(\underline{E}\Gamma;\mathbf{K}_R) \rightarrow \mathcal{H}^{\G}_n(\underline{\underline{E}}\Gamma;\mathbf{K}_R)$. Hence, there is an isomorphism

\[   \mathcal{H}^{\G}_n(\underline{\underline{E}}\Gamma;\mathbf{K}_R)\cong  \mathcal{H}^{\G}_n(\underline{E}\Gamma;\mathbf{K}_R)\oplus \mathcal{H}^\Gamma_n(\underline{\underline{E}}\G,\underline{E}\G;\mathbf{K}_R). \]

If $\Gamma=\mathrm{SL}(3,\mathbb{Z})$, the term $\mathcal{H}^\Gamma_n(\underline{E}\G;\mathbf{K}_R)$ can be computed using a 3-dimensional cocompact model for $\underline{E}\G$ constructed by Soul\'{e} (see \cite{Soule}).  In Theorem, \ref{kth} we describe the term $\mathcal{H}^\Gamma_n(\underline{\underline{E}}\G,\underline{E}\G;\mathbf{K}_R)$ using a $4$-dimensional model for $\underline{\underline{E}}\G$ we construct in Section 5.

\section{A push-out construction}
A general method to obtain a model for $\underline{\underline{E}}\G$ from a model for $\underline{E}\G$ is given by L\"{u}ck and Weiermann in \cite[$\S 2$]{LuckWeiermann}. We will briefly recall this method.

Let $\Gamma$ be a discrete group and consider the set $\mathcal{S}$ of all infinite virtually cyclic subgroups of $\Gamma$. Two infinite virtually cyclic subgroup of $\Gamma$ are said to be equivalent if they have infinite intersection in $\Gamma$. One easily verifies that this defines an equivalence relation on $\mathcal{S}$. If $H \in \mathcal{S}$, then its equivalence class will be denoted by $[H]$. The set of all equivalence classes of elements of $\mathcal{S}$ will be denoted by $[\mathcal{S}]$. Note that the conjugation action of $\Gamma$ on $\mathcal{S}$ passes to $[\mathcal{S}]$. The stabilizer of an $[H] \in \mathcal{S}$ under this action is the subgroup 
\[ \mathrm{N}_{\Gamma}[H]=\{g \in \Gamma \ | \ |H\cap H^{g}|=\infty \} \]
of $\Gamma$. By definition, $\mathrm{N}_{\Gamma}[H]$ only depends on the equivalence class of $[H]$ of $H$. We may therefore always assume that $[H]$ is represented by an infinite cyclic group $H=\langle t \rangle$. Hence, one can write
\[ \mathrm{N}_{\Gamma}[H]=\{g \in \Gamma \ | \ \exists n,m \in \mathbb{Z}\smallsetminus \{0 \}: g^{-1}t^ng=t^m \}.\]
This group is called the commensurator of $H$ in $\Gamma$. Some references, e.g.\cite{DP}, actually denote this group by $\mathrm{Comm}_{\Gamma}[H]$ instead of $\mathrm{N}_{\Gamma}[H]$. Note that $\mathrm{N}_{\Gamma}[H]$ always contains $H$ as a subgroup. 

Let  $\mathcal{I}$ be a complete set of representatives $[H]$ of the orbits of the conjugation action of $\Gamma$ on $[\mathcal{S}]$.
For each $[H] \in \mathcal{I}$,  let $\mathcal{F}[H]$ be the family of subgroups of $\mathrm{N}_{\Gamma}[H]$  containing all finite subgroup of $N_\Gamma[H]$ and all infinite virtually cyclic subgroup of $N_\Gamma[H]$ that are equivalent to $H$.
\begin{theorem}[{\cite[Theorem 2.3]{LuckWeiermann}}]\label{th: push out}  Let 
\[  \xymatrix{ \coprod_{[H] \in \mathcal{I}} \Gamma \times_{\mathrm{N}_{\Gamma}[H]}\underline{E}\mathrm{N}_{\Gamma}[H] \ar[r]^{\ \ \ \ \ \ \ \ \ \ \ \ \  i} \ar[d]_{\coprod_{[H] \in \mathcal{I}} \mathrm{id}_\Gamma \times f_{[H]}} & \underline{E}\Gamma  \ar[d] \\
              \coprod_{[H] \in \mathcal{I}}  \Gamma \times_{\mathrm{N}_{\Gamma}[H]}E_{\mathcal{F}[H]}\mathrm{N}_{\Gamma}[H] \ar[r] & Y } \]
be a $\Gamma$-equivariant push-out diagram of $\Gamma$-$CW$-complexes such that for each $[H] \in \mathcal{I}$, the map
 $f_{[H]}$ is cellular and $\mathrm{N}_{\Gamma}[H]$-equivariant and $i$ is a cellular inclusion of $\Gamma$-$CW$-complexes. Then the push-out $Y$ is a model for $\underline{\underline{E}}\Gamma$.
\end{theorem}
Using {\cite[Remark 2.5]{LuckWeiermann}}, one arrives at the following corollary.
\begin{corollary}[{\cite[Remark 2.5]{LuckWeiermann}}] \label{cor: push out} If there exists a natural number $d$ such that for each $[H] \in \mathcal{I}$
\begin{itemize}
\item[-] $\underline{\mathrm{gd}}(\mathrm{N}_{\Gamma}[H]) \leq d-1$,
\item[-] $\mathrm{gd}_{\mathcal{F}[H]}(\mathrm{N}_{\Gamma}[H]) \leq d$,
\end{itemize}
and such that $\underline{\mathrm{gd}}(\Gamma) \leq d$, then $\underline{\underline{\mathrm{gd}}}(\Gamma) \leq d$.
\end{corollary}
\section{On the structure of $N_{\G}[H]$ in linear groups}
We recall that a {\it real algebraic group} is the set of real points of a linear algebraic group defined over $\mathbb{R}$. Throughout this section, we will use some basic facts about (real) algebraic groups for which we refer to \cite{Borel}.
For  any subgroup $K$ of $\mathrm{GL}(n,\mathbb{R})$, we will denote the Zariski closure of $K$ in $\mathrm{GL}(n,\mathbb{R})$ by $\overline{K}$.  The notions ``connected'' and ``discrete'' will refer to the Hausdorff topology and not to the Zariski topology.

The following result was kindly communicated to us by Herbert Abels. 

\begin{proposition}[H.~Abels] \label{abels} Let $G$ be a real algebraic group and suppose $R$ is its algebraic radical. Suppose $\G$ is a discrete subgroup of $G$ such that the $\pi(\G)$ is Zariski dense in $G/R$, where $\pi: G\to G/R$ is the natural quotient map.  Then $\pi(\G)$ is discrete.
\end{proposition}
\begin{proof} Denote by $C$ the identity component of the closure of the group $\pi(\G)$ in $G/R$ in the Hausdorff topology. We need to show that $C$ is trivial. By Corollary 1.3 of \cite{Abels}, it is solvable. Since $C$ is normalised by $\pi(\G)$, it is also normalised by its Zariski closure $G/R$. We obtain that $C$ is a (Hausdorff and hence Zariski) connected solvable normal subgroup of the semisimple group $G/R$ and hence it is trivial.
\end{proof}

Now, let us assume that $\G$ is a discrete subgroup of $\mathrm{GL}(n,\mathbb{R})$ and let $[H]$ be an equivalence class of infinite virtually cyclic subgroups of $\Gamma$.
\begin{lemma}\label{lemma: norm} There is a representative $H \in [H]$ such that $\overline{H}$ is a Zariski connected abelian normal subgroup of $\overline{N_\G[H]} \leq \mathrm{GL}(n,\mathbb{R})$ and $N_\G[H]=N_\G(H)$, the normaliser of $H$ in $\Gamma$.

\end{lemma}
\begin{proof} Let $H \in [H]$. An algebraic group has only finitely many Zariski connected components. Up to passing to a finite-index subgroup of $H$ we may therefore assume that $\overline{H}$ is a Zariski connected algebraic group. Moreover, since $H$ is abelian, so is $\overline{H}$. Let $x \in N_{\Gamma}[H]$. By definition, $H^{x}\cap H$ is a finite index subgroup of $H$. This implies that $\overline{H}^{x}\cap \overline{H}$ is an algebraic subgroup of $\overline{H}$ of the same dimension. Because $\overline{H}$ is Zariski connected, we conclude that $\overline{H}^{x}=\overline{H}$. Since $N_{\Gamma}[H]$ normalizes $\overline{H}$, it follows that $\overline{H}$ is normal in $\overline{N_{\Gamma}[H]}$. From this we deduce that  $\overline{H}\cap N_{\Gamma}[H]$ is a normal abelian subgroup of $N_{\Gamma}[H]$. Since every discrete subgroup of a finite dimensional abelian Lie group is finitely generated (e.g. see \cite[Proposition 3.8]{Raghunathan}), the structure theorem of finitely generated abelian groups implies that up to passing to a finite-index subgroup, $H$ is contained in an finite rank free abelian  subgroup $A$ of $\overline{H}\cap N_{\Gamma}[H]$ that is normal in $N_{\Gamma}[H]$. But this implies that $H$ is also normal in $N_{\G}[H]$. Indeed, take $g \in N_G[H]$. Since $A$ is normal in $N_{\G}[H]$, conjugation by $g$ induces an automorphism $\varphi$ of $A$. Note that $H$ has a finite index infinite cyclic overgroup  $H'$ in $A$ that has a primitive generator $h$, meaning that $h$ is not a proper power of any other element in $A\cong \mathbb{Z}^r$. Because $g \in N_{\G}[H]=N_{\G}[H']$, there exists  $s, t \in \mathbb{Z}\smallsetminus \{0\}$ such that $s\varphi(h)=th$. Since $\varphi(h)$ is also a primitive element, it now follows that $s$ must divide $t$ and vice versa. Hence $s=\pm t$. It follows that $H'$ is normal in $N_\G[H]$ and so is $H$.
\end{proof}

We continue assuming that $\overline{H}$ is a Zariski connected abelian normal subgroup of $\overline{N_{\G}[H]}$ and $N_\G[H]=N_\G(H)$. Let $m$ be the dimension of $\overline{\G}$. 

Recall also that the notion of Hirsch length $h(S)\in \Z_{\geq 0}$ is defined for all virtually solvable groups $S$. The Hirsch length  is stable under passing to finite index subgroups. It behaves additively with respect to group extensions of virtually solvable groups and satisfies $h(\mathbb{Z})=1$. It also satisfies the relation $h(S)=\sup \{ h(S') \; | \; S' \mbox{ is a finitely generated subgroup of } S\}$. 

\begin{proposition} \label{prop: exact seq}There exists a short exact sequence
\[    1 \rightarrow N \rightarrow N_{\G}[H]/H \rightarrow Q \rightarrow 1   \]
where $Q$ is a discrete subgroup of a $k$-dimensional semisimple algebraic group and $N$ is a finitely generated solvable group of Hirsch length $h(N)\leq m-k-1$.

\end{proposition}
\begin{proof} Let $R$ be the algebraic radical of the Zariski closure $\overline{N_\G[H]} \leq \mathrm{GL}(n,\mathbb{R})$. There is a short exact sequence
\[    1 \rightarrow R \rightarrow \overline{N_\G[H]} \xrightarrow{\pi} S \rightarrow 1    \] where $S=\overline{N_\G[H]}/R$ is semisimple.
Since $N_{\Gamma}[H]$ is Zariski dense in $\overline{N_\G[H]}$ we conclude that $Q= \pi(N_\G[H])$ is Zariski dense in $S$. Hence, by Proposition \ref{abels} it follows that $\pi(N_\G[H])$ is a discrete subgroup of the semisimple real algebraic group $S$. Since the Zariski connected abelian normal subgroup $\overline{H}$ of  $\overline{N_\G[H]}$ has finitely many Hausdorff connected components, up to passing to a finite-index subgroup, we may assume that $H$ is contained in $R$. Denoting $N=(R\cap N_\G[H])/{H}$, we obtain a short exact sequence 
\[    1 \rightarrow N \rightarrow N_\G[H]/H \rightarrow Q \rightarrow 1.   \]
Since every discrete subgroup of a connected solvable  Lie group is finitely generated (see \cite[Proposition 3.8]{Raghunathan}) and $R$, being an algebraic group, has finitely many connected components, it follows that $N$ is a finitely generated solvable group. Suppose $S$ had dimension $k$. Then $R$ has dimension $m-k$. The dimension of $R$ is an upper bound for $\underline{\mathrm{gd}}(R\cap N_\G[H])$ (e.g. see \cite[Theorem 4.4]{Luck2}). Moreover, $\underline{\mathrm{gd}}(R\cap N_\G[H])$ is bounded from below by the Hirsch length of $R\cap N_\G[H]$, since the Hirsch length of a solvable group coincides with its rational homological dimension (see \cite{Stammbach}) . It follows that $h(R\cap N_\G[H]) \leq m-k$. Hence, the Hirsch length of $N$ is at most $m-k-1$. 
\end{proof}
\section{The proofs of the main theorems}


We are now ready to prove Theorems A and B and their Corollary.
\begin{proof}[Proof of Theorem A]
Since the dimension of the Zariski closure  of $\Gamma$ in $\mathrm{GL}(n,\mathbb{R})$ is an upper bound for $\underline{\mathrm{gd}}(\Gamma)$ (e.g. see \cite[Theorem 4.4]{Luck2}), we have $\underline{\mathrm{gd}}(N_\G[H])\leq \underline{\mathrm{gd}}(\Gamma)\leq m$ for every infinite cyclic subgroup $H$ of $\G$.  Now fix $[H] \in \mathcal{I}$ and consider the exact sequence
\[  1 \rightarrow N \rightarrow N_{\G}[H]/H \rightarrow Q \rightarrow 1   \]
resulting from Proposition \ref{prop: exact seq}. By \cite[Lemma 4.2]{DP}, we have $\mathrm{gd}_{\mathcal{F}[H]}(N_{\G}[H])=\underline{\mathrm{gd}}(N_{\G}[H]/H)$ for every $[H] \in \mathcal{I}$. If the Hirsch length of $N$ is at most $1$, then since $N$ is finitely generated it follows that $N$ is virtually cyclic. So,  every finite extension $T$ of $N$ must be virtually cyclic as well. In this case one has $\underline{\mathrm{gd}}(T)\leq 1$. If the Hirsch length of $N$ is a least 2, then it follows from \cite[Corollary 4]{FloresNuc} that every finite extension $T$ of $N$ has $\underline{\mathrm{gd}}(T)\leq h+1$. Since $\underline{\mathrm{gd}}(S)\leq k $ by \cite[Theorem 4.4]{Luck2}, it follows from \cite[Corollary 2.3]{DP} that $\mathrm{gd}_{\mathcal{F}[H]}(N_{\G}[H])\leq (m-k-1)+1+k=m$ for every $[H]\in \mathcal{I}$. The theorem now follows from Corollary \ref{cor: push out}.
\end{proof}

\begin{proof}[Proof of Theorem B] We may assume that $\G$ is a subgroup of $\mathrm{SL}(n, \mathbb C)$  of integral characteristic. Let $A$ be the finitely generated unital subring of $\mathbb C$ generated by the matrix entries of a finite set of generators of $\G$ and their inverses. Then $\G$ is a subgroup of  $\mathrm{SL}(n, A)$. 

Let $\mathbb F$ denote the quotient field of $A$. Proceeding as in the proof of Theorem~3.3 of \cite{AS}, we obtain an epimorphism $\rho: \G \to H_1\times \dots \times H_r$ such that the kernel $U$ of $\rho$ is a unipotent subgroup of $H$ and for each $1\leq i\leq r$, $H_i$ is a subgroup of some $\mathrm{GL}(n_i, \mathbb F)$ of integral characteristic where the canonical action of $H_i$ on $\mathbb F^{n_i}$ is irreducible.  Following the proof of Proposition~2.3 of \cite{AS}, we have that for each subgroup $H_i$, there exists a finite field extension $L_i$ of $\mathbb Q$ such that $H_i$ is isomorphic to a subgroup $H'_i$ of some $\mathrm{GL}(m_i, L_i)$, which is absolutely irreducible and of integral characteristic. Now, according to the proof of Proposition~2.1 of \cite{AS}, each $H'_i$ embeds as a discrete subgroup of $\mathrm{GL}(m_i, \mathbb R)^{r_i}\times \mathrm{GL}(m_i, \mathbb C)^{s_i}$ for some nonnegative  integers $r_i$ and $s_i$.  So, by Theorem A, we have that $\underline{\underline{\mathrm{gd}}}(H_1\times \dots \times H_r)<\infty$. Applying Corollary~6.1 of \cite{DP}, we obtain that 
$$\underline{\underline{\mathrm{gd}}}(\G)\leq \underline{\underline{\mathrm{gd}}}(H_1\times \dots \times H_r) +h+3$$ where $h$ is the Hirsch length of $U$.
\end{proof}


\begin{proof}[Proof of Corollary.]As in the proof of Theorem B, $\G$ fits into an extension $$1\to U\to \G \to H_1\times \dots \times H_r\to 1$$ where $U$ is a unipotent subgroup and for each $1\leq i \leq r$, $H_i$ is a subgroup of integral characteristic of some $\mathrm{GL}(n_i, \mathbb F)$ such that the canonical action of $H_i$ on $\mathbb F^{n_i}$ is irreducible. Following the proof of Proposition 2.1 of \cite{AS}, we obtain that each $H_i$ is isomorphic to a discrete subgroup of $\mathrm{GL}(m_i, \mathbb R)^{r_i}\times \mathrm{GL}(m_i, \mathbb C)^{s_i}$ for some positive integers $r_i$ and $s_i$. By considering the upper central series of  the subgroup of strictly upper triangular matrices $\mathrm{Tr}(n,\mathbb F)$ of $\mathrm{GL}(n,\mathbb{F})$ and noticing that the additive group of $\mathbb{F}$ has finite Hirsch length because it is isomorphic to a finite direct product of copies of $(\mathbb{Q},+)$, it follows that $\mathrm{Tr}(n,\mathbb F)$ has finite Hirsch length. Since the subgroup $U$ of $\G$ is conjugate in $\mathrm{GL}(n,\mathbb{C})$ to a subgroup of $\mathrm{Tr}(n,\mathbb F)$, it also has finite Hirsch length. Just as in the proof of Theorem B, we now conclude that 
$\underline{\underline{\mathrm{gd}}}(\G)<\infty$.
\end{proof}

\section{The case of $\mathrm{SL}(3,\mathbb{Z})$}
Consider the group $\mathrm{SL}(3,\mathbb{Z})$ and let $\mathcal{I}$ be a set of representatives of the orbits of the conjugation action of $\mathrm{SL}(3,\mathbb{Z})$ on the set of equivalence classes of infinite virtually cyclic subgroups of $\mathrm{SL}(3,\mathbb{Z})$ (see Section 2). We note that the infinite virtually cyclic subgroups of $\mathrm{SL}(3,\mathbb{Z})$ are listed, up to isomorphism, in \cite{Stamm} and \cite{Upa}. From this classification  it follows that every infinite virtually cyclic subgroup $V$ of $\G$ that is not isomorphic to $\mathbb{Z}$ or $\mathbb{Z}\oplus \mathbb{Z}_2$ fits into a short exact sequence
\[   1 \rightarrow \mathbb{Z}\oplus \mathbb{Z}_2 \rightarrow V \rightarrow \mathbb{Z}_2 \rightarrow 1 .  \]
\begin{definition} \rm  We define the following subsets of $\mathcal{I}$.

\begin{itemize}
\item[(a)] The set $\mathcal{I}_1$ contains all $[H] \in \mathcal{I}$ such that $[H]$ has a representative whose generator has two complex conjugate eigenvalues and one real eigenvalue different from $1$;
\item[(b)] The set $\tilde{\mathcal{I}}_1$ contains all $[H] \in \mathcal{I}$ such that $[H]$ has a representative whose generator  has exactly one eigenvalue that is a root of unity.
\item[(c)] The set $\mathcal{I}_2$ contains all $[H] \in \mathcal{I}$ such that $[H]$ has a representative with a  generator all of whose eigenvalues are real and not equal to $\pm 1$;
\item[(d)] The set $\tilde{\mathcal{I}}_2$ contains all $[H] \in \mathcal{I}$ such that $[H]$ has a representative with a generator all of whose eigenvalues equal $1$ and which cannot be conjugated into the center of the strictly upper triangular matrices in $\mathrm{SL}(3,\mathbb{Z})$.
\item[(e)] The set $\mathcal{I}_3$ contains all $[H]$ such that $[H]$ has a representative with a generator that can be conjugated into the center of the strictly upper triangular matrices in $\mathrm{SL}(3,\mathbb{Z})$.
\end{itemize}
\end{definition}

\begin{lemma} One can write $\mathcal{I}$ as a disjoint union
\[     \mathcal{I}=\mathcal{I}_1 \sqcup  \tilde{\mathcal{I}}_1 \sqcup \mathcal{I}_2 \sqcup  \tilde{\mathcal{I}}_2  \sqcup \mathcal{I}_3 \]
 and the set $\mathcal{I}_3$ contains exactly one element.
\end{lemma}
\begin{proof} This is left as an easy exercise to the reader. 
\end{proof}
The group $\G=\mathrm{SL}(3,\mathbb{Z})$ is a discrete subgroup of $\mathrm{GL}(3,\mathbb{R})$. Hence, we know from Lemma \ref{lemma: norm} that for every equivalence class $[H]$ of infinite virtually cyclic subgroups of $\G$, there exists an representative $H$ such that $N_\G[H]=N_\G(H)$. Using this fact, we will now determine for each $[H] \in \mathcal{I}$ the structure of the group $N_\G[H]$.

\begin{lemma} \label{lem: comm sl}For each $[H] \in \mathcal{I}$, the following holds.

\begin{itemize}
\item[(a)] If $[H] \in \mathcal{I}_1 $, then $N_{\G}[H]\cong \mathbb{Z}_2\oplus \mathbb{Z}$. 
\item[(b)] If $[H] \in \tilde{\mathcal{I}}_1$, then $N_{\G}[H]$ has a subgroup of index at most two isomorphic to $\mathbb{Z}_2 \oplus \mathbb{Z}$ 

\item[(c)] If $[H] \in \mathcal{I}_2$, then $N_{\G}[H]\cong \mathbb{Z}_2\oplus \mathbb{Z}^2$. 

\item[(d)] If $[H] \in \tilde{\mathcal{I}}_2$, then $N_{\G}[H]$ has a subgroup of index at most two isomorphic to $\mathbb{Z}^2$.
\item[(e)]  If $[H] \in \mathcal{I}_3$,  then $N_{\G}[H]$ is isomorphic to $\mathrm{Tr}(3,\mathbb{Z})\rtimes_{\mathbb{\varphi}} \Big(\mathbb{Z}_2\oplus \mathbb{Z}_2\Big)$, where 
\[   \mathrm{Tr}(3,\mathbb{Z})=\Big\langle x,y,z  \ | \ [x,y]=z , [x,z]=e,[y,z]=e  \Big\rangle   \]
is isomorphic to the group of strictly upper triangular integral matrices,
\[\varphi((1,0))(x)=x^{-1},\varphi((1,0))(y)=y^{-1},\varphi((1,0))(z)=z\] and 
\[\varphi((0,1))(x)=x^{-1},\varphi((0,1))(y)=y,\varphi((0,1))(z)=z^{-1}.\]

\end{itemize}
\end{lemma}

\begin{proof}Take $[H] \in \mathcal{I}$ and let $A \in \G$ be an infinite order matrix such that $\mathrm{N}_{\G}[H]=\mathrm{N}_{\Gamma}(H)$, where $\langle A \rangle =H $. Note that we may replace $A$ by a power of $A$ in order to assume that $A$ does not have any eigenvalues that are non-trivial roots of unity.

First assume that $[H] \in \mathcal{I}_1 \sqcup \mathcal{I}_2$. This means that all eigenvalues of $A$ are different from $1$. The characteristic polynomial $p(x)$ of $A$ is therefore irreducible over $\mathbb{Q}$. Indeed, if $p(x)$ was reducible over $\mathbb{Q}$, then $A$ would have a rational eigenvalue $\mu$. But since $A \in \mathrm{SL}(3,\mathbb{Z})$, it follows from the rational root theorem that $\mu = \pm 1$, which is a contradiction. Also note that the normalizer of $H$ must equal the centralizer of $H$. Indeed, an element of the normalizer of $H$ that does not commute with $A$ must send a eigenvector of $A$ with eigenvalue $\mu$ to an eigenvector of $A$ with eigenvalue $\mu^{-1}$, which would imply that $A$ has an eigenvalue equal to $1$. As illustrated for example in \cite[Prop. 3.7]{kks} and \cite[section 4]{mw}, an application of the Dirichlet unit theorem shows that the centralizer $C_{\G}(H)$ of $A$ in $\Gamma$ equals $\mathbb{Z}^{r+s-1}\oplus \mathbb{Z}_2$, where $r$ is the number of real roots of $p(x)$ and $2s$ is the number of complex roots of $p(x)$. Hence, if all eigenvalues of $A$ are real then $N_{\G}(H)\cong \mathbb{Z}^2\oplus \mathbb{Z}_2$ and if $A$ has two complex conjugate eigenvalues then $N_{\G}(A)\cong \mathbb{Z}\oplus \mathbb{Z}_2$. This proves (a) and (c).

Secondly, assume that $[H] \in \tilde{\mathcal{I}}_1$. Then $A$ has exactly one eigenvalue equal to $1$. Hence, $A$ is conjugate in $\mathrm{SL}(3,\mathbb{Z})$ to a matrix of the form 
\[
\left[
\begin{array}{cc}
  1 & a \   \ b  \\ 
  0 & \raisebox{-10pt}{{\Large\mbox{{$M$}}}} \\[-2ex]
   & \\[-1.9ex]
  0 &
\end{array}
\right].
\]
Since we are only interested in the structure of the normalizer of $H$ up to isomorphism, we may as well assume that $A$ is of this form.  If a matrix $B \in \mathrm{SL}(3,\mathbb{Z})$ commutes with $A$ it must preserve the $1$-dimensional eigenspace of $A$ with eigenvalue $1$. Therefore, this is also an eigenspace of $B$, with eigenvalue $\pm 1$. We conclude that $B$ must be of the form
\[
B=\left[
\begin{array}{cc}
  \pm 1 & x \  \ y  \\ 
  0 & \raisebox{-10pt}{{\Large\mbox{{$N$}}}} \\[-2ex]
   & \\[-1.9ex]
  0 &
\end{array}
\right].
\]
By elementary matrix computations, one checks that such a matrix $B$ commutes with $A$ if and only if  $M$ commutes with $N$ and 
\[  (M^t - \mathrm{Id})\left[\begin{array}{c}x \\ y  \end{array}\right] = (N^t - \mathrm{Id})\left[ \begin{array}{c}a \\ b  \end{array}\right] .  \]
Since  $(M^t - \mathrm{Id})$ is an invertible matrix, $B$ is completely determined by $N$ and the fact that its commutes with $A$. We therefore obtain an isomorphism $C_{\G}(A)=C_{\mathrm{GL}(2,\mathbb{Z})}(M)$. By analyzing centralizers in $\mathrm{GL}(2,\mathbb{Z})$, for example using the Dirichlet unit theorem, it follows that the centralizer $C_{\G}(A)$ is  $\mathbb{Z}_2 \oplus \mathbb{Z}$. This proves (b).

Finally, assume that all eigenvalues of $A$ equal $1$. In this case $A$ can be conjugated inside $\mathrm{SL}(3,\mathbb{Z})$ to a strictly upper triangular matrix. Hence, we may again assume that $A$ is a strictly upper triangular matrix. If $A$ is of the form
\begin{equation}\label{eq: upper triangular}
\left[
\begin{array}{ccc}
  1 & a & c \\
  0 &1  & b \\
  0 & 0 & 1 
\end{array}
\right]
\end{equation}
where $a$ and $b$ are both non-zero, then one may check by elementary matrix computations that a matrix $B \in \mathrm{SL}(3,\mathbb{Z})$ commutes with $A$ if and only if it is of the form
\[
\left[
\begin{array}{ccc}
  1 & x & z \\
  0 & 1 & y \\
  0 & 0 & 1 
\end{array}
\right]
\]
where $ay-bx=0$. This shows that in this case the centralizer of $H$ in $\Gamma$ is isomorphic to $\mathbb{Z}^2$, and hence the normalizer $N_{\G}(H)$ has a subgroup of index at most two isomorphic to $\mathbb{Z}^2$. If on the other hand, $A$ is of the form (\ref{eq: upper triangular}) where $ab=0$ then $A$ can be conjugated in $\mathrm{SL}(3,\mathbb{Z})$ to matrix of the form (\ref{eq: upper triangular}) where $a$ and $b$ are both zero. In this case the centralizer $C_{\Gamma}(H)$ is isomorphic to the group 
\[   \left\{ \left[
\begin{array}{ccc}
  \pm 1 & x & z \\
  0 & 1 & y \\
  0 & 0 & \pm 1 
\end{array}
\right] \ \Bigg| \ x,y,z \in \mathbb{Z} \right\} . \]
One can now easily verify via explicit matrix computation that the normalizer $N_{\G}(H)$ is isomorphic to a semi-direct product of $C_{\Gamma}(H)$ with
\[   \left\{  \left[
\begin{array}{ccc}
  \pm 1 & 0 & 0\\
  0 & \pm 1  & 0 \\
  0 & 0 & 1  \end{array}\right] \right\}. \]It follows that $N_{\G}(H)$ is isomorphic to the semi-direct product \[\mathrm{Tr}(3,\mathbb{Z})\rtimes_{\mathbb{\varphi}} \Big(\mathbb{Z}_2\oplus \mathbb{Z}_2\Big).\]
  \end{proof}
  In \cite{Soule}, a 3-dimensional model $X$ for $\underline{E}\G$ is constructed. This model has the property that the orbit-space $\Gamma \setminus X=\underline{B}\G$ is contractible. Moreover, this model is of minimal dimension since $\Gamma$ contains the strictly upper triangular matrices $\mathrm{Tr}(3,\mathbb{Z})$, which has cohomological dimension $3$. Since for each $[H] \in \mathcal{I} $, $N_{\G}[H]$ is either virtually-$\mathbb{Z}$, virtually-$\mathbb{Z}^2$ or virtually-$\mathrm{Tr}(3,\mathbb{Z})$ by the lemma above, a model for $\underline{E}N_{\G}[H]$ can be chosen to be either $\mathbb{R}$, $\mathbb{R}^2$ or $\mathbb{R}^3$, respectively (see, e.g., \cite[Ex.~5.26]{Luck2}). Moreover, $H$ can be chosen to be normal in $N_{\G}[H]$ in which case a model for $E_{\mathcal{F}[H]}N_{\G}[H]$ is given by a model for  $\underline{E}N_{\G}[H]/H$, where the action is obtained via the projection $N_{\G}[H] \rightarrow N_{\G}[H]/H$.  Hence, a model for $E_{\mathcal{F}[H]}N_{\G}[H]$ can be chosen to be either $\{\ast\}$, $\mathbb{R}$ or $\mathbb{R}^2$, respectively. Using the universal property of classifying spaces for families, one obtains a cellular $\G$-equivariant map 
\[   f:  \coprod_{[H] \in \mathcal{I}} \Gamma \times_{\mathrm{N}_{\Gamma}[H]} \underline{E}N_\G[H] \rightarrow X   \]
Note that the mapping cylinder $M_f$ of $f$ is a $4$-dimensional model for $\underline{E}\G$, since $\underline{E}N_\G[H] $ is at most $3$-dimensional and $M_f$ is $\G$-homotopy equivalent to $X$. We  obtain an equivariant cellular inclusion
\[    i:  \coprod_{[H] \in \mathcal{I}} \Gamma \times_{\mathrm{N}_{\Gamma}[H]} \underline{E}N_\G[H] \rightarrow M_f.   \]   
Using $i$ and the models for $E_{\mathcal{F}[H]}N_G[H]$ described above, one can construct a $\G$-equivariant push-out diagram that by Theorem \ref{th: push out} produces a $4$-dimensional model $Y$ for $\underline{\underline{E}}\Gamma$. We claim that this model is of minimal dimension. Indeed, take $\G$-orbits of the push-out diagram constructed above and consider the long exact Mayer--Vietoris cohomology sequence with $\mathbb{Q}$-coefficients obtained from the resulting push-out diagram. This leads to the exact sequence
\[      \mathrm{H}^3(\underline{B}\Gamma,\mathbb{Q}) \rightarrow   \mathrm{H}^3(\underline{B}N_{\G}[H],\mathbb{Q}) \rightarrow  \mathrm{H}^4(\underline{\underline{B}}\Gamma,\mathbb{Q}) \rightarrow 0,       \]
where $[H] \in \mathcal{I}_3$. Since $\underline{B}\Gamma$ is contractible and $N_{\G}[H]\cong \mathrm{Tr}(3,\mathbb{Z})\rtimes_{\mathbb{\varphi}} (\mathbb{Z}_2\oplus \mathbb{Z}_2)$ by the lemma above , we obtain an isomorphism
\[      \mathrm{H}^3\Big(\mathrm{Tr}(3,\mathbb{Z})\rtimes_{\mathbb{\varphi}} (\mathbb{Z}_2\oplus \mathbb{Z}_2),\mathbb{Q}\Big) \cong  \mathrm{H}^4(\underline{\underline{B}}\Gamma,\mathbb{Q}) .    \]
As we are working with $\mathbb{Q}$-coefficients, an application of the Lyndon--Hochschild--Serre spectral sequence  tells us that 
\[  \mathrm{H}^3(\mathrm{Tr}(3,\mathbb{Z})\rtimes_{\mathbb{\varphi}} \Big(\mathbb{Z}_2\oplus \mathbb{Z}_2\Big),\mathbb{Q}) \cong  \mathrm{H}^3(\mathrm{Tr}(3,\mathbb{Z}),\mathbb{Q})^{\mathbb{Z}_2\oplus \mathbb{Z}_2}.  \]
Moreover, since $\mathrm{Tr}(3,\mathbb{Z})$ fits into the central extension
\[  1 \rightarrow \mathbb{Z}\cong\langle z \rangle \rightarrow \mathrm{Tr}(3,\mathbb{Z}) \rightarrow   \mathbb{Z}^2\cong \langle x,y \rangle \rightarrow 0, \]
another application of the Lyndon--Hochschild--Serre spectral sequence yields

\[\mathrm{H}^3(\mathrm{Tr}(3,\mathbb{Z}),\mathbb{Q})^{\mathbb{Z}_2\oplus \mathbb{Z}_2}\cong \mathrm{H}^2(\langle x,y \rangle,\mathrm{H}^1(\langle z \rangle, \mathbb{Q}))^{\mathbb{Z}_2\oplus \mathbb{Z}_2}.  \]
Using the explicit description of the map $\varphi: \mathbb{Z}_2\oplus \mathbb{Z}_2 \rightarrow \mathrm{Aut}(\mathrm{Tr}(3,\mathbb{Z}))$ in the lemma above, and the fact that 
\[\mathrm{H}^2(\langle x,y \rangle,\mathrm{H}^1(\langle z \rangle, \mathbb{Q}))= \mathrm{Hom}(\Lambda^2(\langle x , y \rangle ),\mathrm{Hom}((\Lambda^{1}(\langle z \rangle) ,\mathbb{Q})\cong \mathbb{Q,}\] one checks that the action of $\mathbb{Z}_2\oplus \mathbb{Z}_2 $ on $\mathrm{H}^2(\langle x,y \rangle,\mathrm{H}^1(\langle z \rangle, \mathbb{Q}))$ is trivial. We conclude that $\mathrm{H}^4(\underline{\underline{B}}\Gamma,\mathbb{Q}) \cong \mathbb{Q}$, proving that there cannot exists a model for $\underline{\underline{E}}\G$ of dimension strictly smaller than $4$.\\

As mentioned in the introduction, for $\G=\mathrm{SL}(3,\mathbb{Z})$, the Farrell--Jones conjecture implies that for any ring $R$ that is finitely generated as an abelian group, one has
\[   \mathrm{K}_n(R[\G])\cong  \mathcal{H}^{\G}_n(\underline{E}\Gamma;\mathbf{K}_R)\oplus \mathcal{H}^\Gamma_n(\underline{\underline{E}}\G,\underline{E}\G;\mathbf{K}_R) \]
for every $n \in \mathbb{Z}$. Using the model $Y$ for $\underline{\underline{E}}\G$ constructed above and Lemma \ref{lem: comm sl} we obtain a description of the term $ \mathcal{H}^\Gamma_n(\underline{\underline{E}}\G,\underline{E}\G;\mathbf{K}_R)$.  We summarize this description in the following theorem. Note that given a $\Gamma$-map $f: X \rightarrow Y$, the homology group $\mathcal{H}^\Gamma_n(Y,X;\mathbf{K}_R)$ is by definition the relative homology group  $\mathcal{H}^\Gamma_n(M_f, X;\mathbf{K}_R)$, where $M_f$ is the mapping cylinder of $f$.
\begin{theorem} \label{kth}Let $\G=\mathrm{SL}(3,\mathbb{Z})$ and let $R$ be a ring that is finitely generated as an abelian group. Then,
\[     \mathrm{K}_n(R[\G])\cong\mathcal{H}^\Gamma_n(\underline{E}\G;\mathbf{K}_R)\oplus \mathcal{H}_n(\mathcal{I}_1) \oplus \mathcal{H}_n(\tilde{\mathcal{I}}_1)  \oplus \mathcal{H}_n(\mathcal{I}_2) \oplus \mathcal{H}_n(\tilde{\mathcal{I}}_2)   \oplus \mathcal{H}_n(\mathcal{I}_3) \]
where, 
\begin{itemize}
\item[(a)] for $[H] \in \mathcal{I}_1$, $N_\G[H]\cong \mathbb{Z}_2\oplus \mathbb{Z}$, $\underline{\underline{E}}N_\G[H]=\{\ast\}$ , $\underline{E}N_\G[H]=\mathbb{R}$ and 
\[  \mathcal{H}_n(\mathcal{I}_1) =\bigoplus_{[H] \in \mathcal{I}_1}\mathcal{H}^{N_\Gamma[H]}_n(\{\ast\}, \mathbb{R};\mathbf{K}_R),  \]
\item[(b)] for $[H] \in \tilde{\mathcal{I}}_1$, $N_\G[H]$ has a subgroup of index at most two isomorphic to $\mathbb{Z}_2 \oplus \mathbb{Z},$  $\underline{\underline{E}}N_\G[H]=\{\ast\}$ , $\underline{E}N_\G[H]=\mathbb{R}$ and \[  \mathcal{H}_n(\tilde{\mathcal{I}}_1) =\bigoplus_{[H] \in \tilde{\mathcal{I}}_1}\mathcal{H}^{N_\Gamma[H]}_n(\{\ast\}, \mathbb{R};\mathbf{K}_R),  \]
\item[(c)] for $[H] \in \mathcal{I}_2$, $N_\G[H]\cong \mathbb{Z}_2\oplus \mathbb{Z}^2$, $E_{\mathcal{F}[H]}N_\G[H]=\underline{E}N_\G[H]/H=\mathbb{R}$ , $\underline{E}N_\G[H]=\mathbb{R}^2$ and 
\[  \mathcal{H}_n(\mathcal{I}_2) =\bigoplus_{[H] \in \mathcal{I}_2}\mathcal{H}^{N_\Gamma[H]}_n(\mathbb{R}, \mathbb{R}^2;\mathbf{K}_R),  \]
\item[(d)] for $[H] \in \tilde{\mathcal{I}_2}$, $N_\G[H]$ has a subgroup of index at most two isomorphic to $ \mathbb{Z}^2$, $E_{\mathcal{F}[H]}N_\G[H]=\underline{E}N_\G[H]/H=\mathbb{R}$ , $\underline{E}N_\G[H]=\mathbb{R}^2$ and 
\[  \mathcal{H}_n(\tilde{\mathcal{I}}_2) =\bigoplus_{[H] \in \tilde{\mathcal{I}}_2}\mathcal{H}^{N_\Gamma[H]}_n(\mathbb{R}, \mathbb{R}^2;\mathbf{K}_R),  \]
\item[(e)] for $[H] \in \tilde{\mathcal{I}_3}$, $N_\G[H]\cong \mathrm{Tr}(3,\mathbb{Z})\rtimes_{\mathbb{\varphi}} \Big(\mathbb{Z}_2\oplus \mathbb{Z}_2\Big), E_{\mathcal{F}[H]}N_\G[H]=\underline{E}N_\G[H]/H=\mathbb{R}^2$ , $\underline{E}N_\G[H]=\mathbb{R}^3$ and 
\[  \mathcal{H}_n(\mathcal{I}_3) =\mathcal{H}^{N_\Gamma[H]}_n(\mathbb{R}^2, \mathbb{R}^3;\mathbf{K}_R).  \]
\end{itemize}
\end{theorem}

\begin{center}\textbf{Acknowledgement}\end{center}
We would like to thank Herbert Abels for  helpful communications  that  led to the proof of Proposition \ref{abels}. The second-named author expresses his gratitude to the research group Algebraic Topology \& Group Theory at KU Leuven, Campus Kortrijk for their hospitality in spring 2013.

\end{document}